\documentclass{article}
\usepackage{geometry}               % See geometry.pdf to learn the layout options. There are lots.
\usepackage[english]{babel}
\usepackage[latin1]{inputenc}
\usepackage{amssymb,amsmath,amsthm,amsfonts}
\usepackage{graphicx}
\usepackage[usenames,dvipsnames]{xcolor}
\usepackage{cancel}
\usepackage[colorlinks]{hyperref}
\usepackage{filecontents}
\usepackage[normalem]{ulem}
\usepackage{dsfont}
\usepackage[showonlyrefs]{mathtools}
\mathtoolsset{showonlyrefs=true}
\usepackage{newtxtext,newtxmath,helvet,courier}
\usepackage{fix-cm}
\usepackage{tikz}
\newcommand{\grigetto}{Periwinkle}
\newcommand{\R}{\mathbb{R}}

\newcommand{\eps}{\varepsilon}

\newcommand{\la}{\lambda}

\newcommand{\into}{\int_{\Omega}}

\newcommand{\Om}{\Omega}

\newcommand{\Rcal}{{\mathcal{R}}}

\newcommand{\Vcal}{{\mathcal{V}}}

\def\XXint#1#2#3{{\setbox0=\hbox{$#1{#2#3}{\int}$ }
\vcenter{\hbox{$#2#3$ }}\kern-.6\wd0}}

\newcommand{\ind}[1]{\mathds{1}_{#1}}

\DeclareMathOperator{\dist}{dist}

\newtheorem{proposition}{Proposition}[section]
\newtheorem{theorem}[proposition]{Theorem}

\newtheorem{lemma}[proposition]{Lemma}

\theoremstyle{definition}
\newtheorem{definition}[proposition]{Definition}
\newtheorem{remark}[proposition]{Remark}
\newcommand{\beq}{\begin{equation}}
\newcommand{\eeq}{\end{equation}}
\newcommand{\ben}{\begin{enumerate}}
\newcommand{\een}{\end{enumerate}}
\newcommand{\bit}{\begin{itemize}}
\newcommand{\eit}{\end{itemize}}

\DeclareMathOperator{\sd}{M}
\DeclareMathOperator{\od}{\Lambda}

\title
{Quantitative analysis of a singularly perturbed shape optimization problem in a polygon}

\author{Dario Mazzoleni, Benedetta Pellacci and Gianmaria Verzini}

\begin{document}
\maketitle
\begin{abstract}
We carry on our study of the connection between two shape optimization problems with spectral cost. 
On the one hand, we consider the optimal design problem for the survival threshold of a population living in a heterogenous  habitat $\Omega$;
this problem arises when searching for the optimal shape and location of a shelter zone in order to prevent extinction of the species. On the other hand, we 
deal with the spectral drop problem, which consists in minimizing a mixed Dirichlet-Neumann 
eigenvalue in a box $\Omega$. In a previous paper \cite{mapeve} we proved that the latter one can be 
obtained as a singular perturbation of the former, when  the region outside the refuge is 
more and more hostile. In this paper we sharpen our analysis in case $\Omega$ is a planar polygon, 
providing quantitative estimates of the optimal level convergence, as well as of the involved 
eigenvalues. 
\end{abstract}
\noindent
{\footnotesize \textbf{AMS-Subject Classification}}. 
{\footnotesize 49R05, 49Q10; 92D25, 35P15, 47A75.}\\
{\footnotesize \textbf{Keywords}}. 
{\footnotesize Singular limits, survival threshold, mixed Neumann-Dirichlet boundary conditions, $\alpha$-symmetrization, isoperimetric profile.}
%%%%%%%%%\section{Introduction}\label{sec:intro}
\section{Introduction}\label{sec:intro}

In this note we investigate some relations between the two following shape 
optimization problems, settled in a box $\Omega\subset\R^N$, that is, a bounded, Lipschitz domain (open and connected).
\begin{definition}\label{def:lambda}
Let $0<\delta<|\Omega|$ and $\beta>\dfrac{\delta}{|\Omega|-\delta}$. For any measurable 
$D\subset\Omega$ such that $|D| = \delta$, we define the \emph{weighted eigenvalue}
\begin{equation}\label{eq:def_lambda_beta_D}
\lambda(\beta,D):=\min \left\{
\dfrac{\int_\Omega |\nabla u|^2\,dx}{\int_D u^2\,dx - \beta \int_{\Omega\setminus D} u^2\,dx} :  u\in H^1(\Omega),\ \int_D u^2\,dx>\beta \int_{\Omega\setminus D} u^2\,dx\right\},
\end{equation}
and the \emph{optimal design problem for the survival threshold
 as}
\begin{equation}\label{eq:def_od}
\od(\beta,\delta)=\min\Big\{\lambda(\beta,D):D\subset \Om,\ |D|=\delta\Big\}.
\end{equation}
\end{definition}
\begin{definition}
Let $0<\delta<|\Om|$. Introducing the space
$H^1_0(D,\Om):=\left\{u\in H^1(\Om):u=0\text{ q.e. on }\Om\setminus 
D\right\}$ (where q.e. stands for quasi-everywhere, i.e. up to sets of zero 
capacity), we can define, 
for any quasi-open $D\subset\Omega$ such that $|D| = \delta$,  the 
\emph{mixed Dirichlet-Neumann eigenvalue} as
\begin{equation}\label{eq:def_mu_D}
\mu(D,\Om):=\min{\left\{\frac{\int_{\Om}|\nabla u|^2\,dx}{\int_\Om u^2\,dx}:u\in H^1_0(D,\Om)\setminus\{0\}\right\}},
\end{equation}
and  the \emph{spectral drop problem} as
\begin{equation}\label{eq:def_sd}
\sd(\delta)=\min{\Big\{\mu(D,\Omega):D\subset \Om,\;\mbox{quasi-open, }|D|=\delta\Big\}}.
\end{equation}
\end{definition}
The two problems above have been the subject of many investigations in the literature. 
The interest in the study of the eigenvalue $\lambda(\beta,D)$ goes 
back to the analysis of the optimization of the survival threshold of a species
living in a heterogenous habitat $\Omega$, with the boundary $\partial\Omega$ acting as 
a reflecting barrier. As explained by Cantrell and Cosner in a series of paper
\cite{MR1014659,MR1105497,MR2191264} (see also \cite{ly,llnp,mapeve}),
the heterogeneity of $\Omega$ makes the intrinsic growth rate of the 
population, represented by a $L^{\infty}(\Omega)$ function $m(x)$, be 
positive in favourable sites and negative in the hostile ones. Then,
if $m^{+}\not\equiv 0$ and $\int m<0$, it turns out that the positive principal 
eigenvalue $\lambda=\lambda(m)$ of the problem
\[
\begin{cases}
-\Delta u = \lambda m u &\text{in }\Omega\\
\partial_\nu u = 0 &\text{on }\partial\Omega,
\end{cases}
\]
i.e.
\[
\lambda(m)=\left\{\frac{\into |\nabla u|^{2}dx}{\into mu^{2}dx}: u\in H^{1}(\Omega), \into mu^{2}dx>0\right\},
\]
acts a survival threshold, namely the smaller $\lambda(m)$ is, the greater the 
chances of survival become. Moreover,  by \cite{ly}, the minimum of $\lambda(m)$ w.r.t. $m$
varying in a suitable class is achieved when $m$ is of bang-bang type, i.e.
$m=\ind{D} -\beta \ind{\Omega\setminus D}$, being $D\subset \Omega$
with fixed measure. As a consequence, one is naturally led to the shape optimization
problem introduced in Definition \ref{def:lambda}.

On the other hand, the spectral drop problem has been introduced in \cite{buve} as
a class of shape optimization problems where one minimizes the first
eigenvalue $\mu=\mu(D,\Om)$ of the Laplace operator with homogeneous Dirichlet conditions
on $\partial D\cap \Omega$ and homogeneous Neumann ones on
$\partial D\cap \partial \Omega$:
\[
\begin{cases}
-\Delta u = \mu u &\text{in }D\\
u = 0 &\text{on }\partial D\cap\Omega\\
\partial_\nu u = 0 &\text{on }\partial D\cap\partial\Omega.
\end{cases}
\]

In our paper \cite{mapeve}, we analyzed the relations between the above problems, showing in 
particular that $\sd(\delta)$ arises from $\od(\beta,\delta)$ in the singularly 
perturbed limit  $\beta\to+\infty$, as stated in the following result.
\begin{theorem}[{\cite[Thm. 1.4, Lemma 3.3]{mapeve}}]\label{thm:convergence}
If $0<\delta<|\Omega|$, $\beta>\dfrac{\delta}{|\Omega|-\delta}$ and $\dfrac{\delta}{\beta}<\eps<
|\Omega|-\delta$ then 
\[
\sd(\delta+\eps)\left(1-\sqrt{\frac{\delta}{\eps\beta}}\right)^{2}\leq \od(\beta,\delta)\leq \sd(\delta).
\]
As a consequence, for every $0<\delta<|\Omega|$,
\[
\lim_{\beta\to+\infty} \od(\beta,\delta) = \sd(\delta).
\]
\end{theorem}
In respect of this asymptotic result, let us also mention \cite{derek}, where
the relation between the above eigenvalue problems has been recently investigated  for 
$D\subset\Omega$ fixed and regular.

In \cite{mapeve}, we used the theorem above to transfer information from 
the  spectral drop problem to the optimal design one.
In particular,  we could give a contribution in the comprehension of the shape
of an optimal set $D^{*}$ for  $\od(\beta,\delta)$. This topic includes several 
open questions starting from the analysis performed in \cite{MR1105497} (see also
\cite{llnp,ly}) when $\Omega=(0,1)$: in this  case it is shown that
any optimal set $D^{*}$ is either $(0,\delta)$ or $(1-\delta,1)$.
The knowledge of analogous  features in the higher dimensional case is far from being well understood, but it has been recently proved in \cite{llnp} that when $\Omega$ is an 
N-dimensional rectangle, then $\partial D^{*}$ does not contain any portion of 
sphere, contradicting previous conjectures and numerical studies \cite{MR2214420,haro,MR2494032}.
This result prevents  the existence of optimal \emph{spherical shapes},
namely optimal $D^{*}$ of the form $D^{*}=\Omega\cap B_{r(\delta)}(x_{0})$
for suitable $x_{0}$ and $r(\delta)$ such that $|D^{*}|=\delta$.

On the other hand, we have shown that spherical shapes are optimal for 
$\sd(\delta)$, for small $\delta$, when  $\Omega$ is an $N$-dimensional polytope. This,
together with Theorem \ref{thm:convergence}, yields the following result.
\begin{theorem}[{\cite[Thm. 1.7]{mapeve}}]\label{thm:orthotope}
Let $\Omega \subset \R^N$ be a bounded, convex polytope. There exists $\bar\delta>0$ such that, for any $0<\delta< \bar\delta$:
\begin{itemize}
\item $D^*$ is a minimizer of the spectral drop problem
in $\Omega$, with volume constraint $\delta$, if and only if
$D^*=B_{r(\delta)}(x_0)\cap\Omega$, where $x_0$ is a vertex of $\Omega$ with the smallest solid angle;
\item if $|D|=\delta$ and $D$ is not a spherical shape as above, then, for $\beta$ sufficiently large,
\[
\lambda(\beta,D)> \lambda (\beta,B_{r(\delta)}(x_0)\cap \Omega).
\]
\end{itemize}
In particular, in case $\Omega = (0,L_1) \times (0,L_2)$, with $L_1\le L_2$, and $0<\delta< L_1^2/\pi$, then any minimizing
spectral drop is a quarter of a disk centered at a vertex of $\Omega$.
\end{theorem}
Then, even though the optimal shapes for $\Lambda(\beta,\delta)$ 
can not be spherical for any fixed $\beta$, they are asymptotically spherical as $\beta\to+\infty$, 
at least in the qualitative sense described in Theorem \ref{thm:orthotope}.

The main aim of the present note is to somehow revert the above point of view: we will show that, in 
case $\sd(\delta)$ is explicit as a function of $\delta$, one can use Theorem \ref{thm:convergence}
in order to obtain quantitative bounds on the ratio
\[
\frac{\od(\beta,\delta)}{\sd(\delta)}.
\]
In particular, we will pursue this program in case $\Omega$ is a planar polygon:
indeed, on the one hand, in such case the threshold $\bar\delta$ in Theorem \ref{thm:orthotope} can 
be estimated explicitly; on the other hand, such theorem implies that the optimal shapes for 
$\sd(\delta)$ are spherical, so that $\sd(\delta)$ can be explicitly computed. This will lead to quantitative 
estimates about the convergence of $\od(\beta,\delta)$ to $\sd(\delta)$.

As a byproduct of this analysis, we will also obtain some quantitative information on the 
ratio 
\[
\frac{\lambda(\beta,B_{r(\delta)}(p)\cap \Omega)}{\od(\beta,\delta)},
\]
thus providing a quantitative version of the second part of Theorem \ref{thm:orthotope}.

These new quantitative estimates are the main results of this note, and they are contained in 
Theorems \ref{thm:main1} and \ref{thm:main2}, respectively. The next section is devoted to their 
statements and proofs, together with further details of our analysis.

%%%%%%%%%\section{results}\label{sec:res}
%%%%%%%%%\section{results}\label{sec:res}
\section{Setting of the problem and  main results.}\label{sec:poli}

Let $\Omega\subset\R^2$ denote a convex $n$-gon, $n\ge 3$. We introduce the following quantities and objects, all depending on $\Omega$:
\begin{itemize}
 \item $\alpha_{\min}$ is the smallest interior angle;
 \item $\Vcal_{\min}$ is the set of vertices having angle $\alpha_{\min}$;
 \item $e_1,\dots,e_n$ are the (closed) edges;
 \item $d$ denotes the following quantity: 
\[
d=\min\{\dist(e_i\cap e_j,e_k) : i\neq j,\ i\neq k,\ j\neq k\}.
\]
\end{itemize}
Under the above notation, we define the threshold
\begin{equation}\label{eq:deltabar}
\bar \delta:=\dfrac{d^2}{2\alpha_{\min}}. 
\end{equation}
\begin{remark}
Notice that, as far as $n\ge4$, $d$ corresponds to the shortest distance between two non-
consecutive edges:
\[
d=\min\{\dist(x_{i},x_{j}) : x_{i}\in e_{i}, \, x_{j}\in e_{j},\, e_{i}
\cap e_{j}=\emptyset\}.
\] 
Moreover, for any $n$,
\[
0<\bar\delta <|\Omega|.
\]
Indeed, let $e_i\cap e_j\in\Vcal_{\min}$, with $|e_i|\le|e_j|$. Then
\[
d\le |e_i| \sin \alpha_{\min}\qquad\text{ and }\qquad |\Omega|\ge \frac12 |e_i||e_j| \sin \alpha_{\min},
\]
and the claim follows since $\sin \alpha_{\min} < \alpha_{\min}$.
\end{remark}
Our main results are the following.
\begin{theorem}\label{thm:main1}
Let $\Omega\subset\R^2$ denote a convex $n$-gon, let $\bar\delta$ be defined in \eqref{eq:deltabar}, 
and let us assume that 
\[
0 < \delta < \bar\delta.
\]
Then $\sd(\delta)$ is achieved by $D^{*}$ if and only if $D^*=B_{r(\delta)}(p)\cap 
\Omega$, where $p\in \Vcal_{\min}$. Moreover
\[
\beta > \max\left\{\left(\frac{\delta}{\bar\delta-\delta}\right)^3,1\right\}
\qquad\implies\qquad
(1+\beta^{-1/3})^{-1}\left(1-\beta^{-1/3}\right)^{2}<\frac{\od(\beta,\delta)}{\sd(\delta)}<1.
\]
\end{theorem}
By taking advantage of the asymptotic information on $\od(\beta,\delta)/\sd(\delta)$, we can deduce the corresponding relation between the eigenvalue of a spherical shape and the minimum $\od(\beta,\delta)$.  
\begin{theorem}\label{thm:main2}
Let $\Omega\subset\R^2$ denote a convex $n$-gon, $\beta>1$, 
and let us assume that 
\begin{equation}\label{eq:assdelta}
\delta <\frac{\beta^{1/3}}{\beta^{1/3}+1}\, \bar \delta,
\end{equation}
where $\bar\delta$ is defined in \eqref{eq:deltabar}. Then, taking $p\in \Vcal_{\min}$ and
$r(\delta)$ such that $|B_{r(\delta)}(p)\cap 
\Omega|=\delta$,
\[
1 < \frac{\la(\beta,B_{r(\delta)}(p)\cap \Omega)}{\Lambda(\beta,\delta)}< \left(1+\beta^{-\frac{1}{3}}\right)\left(1-\beta^{-\frac{1}{3}}\right)^{-2}.
\]
\end{theorem}
To prove our results, we will use the analysis we developed in \cite[Section~4]{mapeve} to estimate 
$\sd(\delta)$ by means of $\alpha$-symmetrizations on cones \cite{MR876139,MR963504}. To this aim we 
will first evaluate a suitable isoperimetric constant. 

For $D\subset\Omega$, we write
\[
\Rcal(D,\Omega):=\frac{P(D,\Omega)}{2|D\cap\Omega|^{1/2}},
\]
where $P$ denotes the relative De Giorgi perimeter. For $0<\delta<|\Omega|$ we consider 
the isoperimetric problem
\[
I(\Omega,\delta) := \inf\left\{\Rcal(D,\Omega) : D\subset\Omega,\ |D|= \delta\right\},
\]
and we call \[
K(\Omega,\delta) = \inf_{0<\delta'\le\delta} I(\Omega,\delta').
\]
Given the unbounded cone with angle $\alpha$,
\[
\Sigma_\alpha :=\{(r\cos\vartheta,r\sin\vartheta)\in\R^2 : 0<\vartheta<\alpha,\; r>0\},
\]
it is well known that
\begin{equation}\label{eq:sector}
I(\Sigma_\alpha,\alpha r^2/2)=\Rcal(B_r(0)\cap\Sigma_\alpha,\Sigma_\alpha) = \frac{\alpha r}{2|\alpha r^2/2|^{1/2}} =
\sqrt{\frac{\alpha}{2}},
\end{equation}
is independent on $r$, and hence on $\delta=|B_r(0)\cap\Sigma_\alpha|$. As a consequence, also
\[
K(\Sigma_\alpha,\delta)=\sqrt{\frac{\alpha}{2}},
\]
for every $\delta$.
\begin{lemma}\label{le:isop}
If $\Omega\subset\R^2$ is a convex $n$-gon and $\delta < \bar \delta$, then $I(\Omega,\delta)$ is achieved by $D^*$ if and only if $D^*=B_{r(\delta)}(p)\cap  \Omega$, where $p\in \Vcal_{\min}$. Moreover $K(\Om,\bar \delta)$ is achieved by the same $D^*$ too.
\end{lemma}
\begin{proof}
Notice that, by assumption, for any $p\in \Vcal_{\min}$ the set $D=B_{r(\delta)}(p)\cap 
\Omega$ is a circular sector of measure $\delta$, with $\partial D \cap \Omega$ a circular arc. Then \eqref{eq:sector}
implies
\begin{equation}\label{eq:lemmino}
I(\Omega,\delta) \le I(\Sigma_{\alpha_{\min}},\delta) = \sqrt{\frac{\alpha_{\min}}{2}},
\end{equation}
and we are left to show the opposite inequality (strict, in case $D$ is not of the above kind). 
Applying Theorems 4.6 and 5.12 in \cite{MR3335407}, and
Theorems 2 and 3 in \cite{cianchi}, we deduce that $I$ is
achieved by $D^*_\delta\subset\Omega$, which is an
open, connected set, such that $\Gamma:=\partial D^*_\delta\cap \Om$ is
either a (connected) arc of circle or a straight line segment. Moreover, 
$\partial D^*_\delta\cap \partial\Omega$ consists in exactly two points (the endpoints of $\Gamma$), 
and $\partial D^*_\delta\cap \Om$ reaches the boundary of $\Om$ orthogonally
at flat points (i.e. not at a vertex). Hence, there are three possible configurations (see Fig. 
\ref{fig:HO}). 
\begin{enumerate}
\item[A.] 
The endpoints of $\Gamma$ belong to the interior of two consecutive edges $e_i$ and $e_{i+1}$.
In this case $\Gamma$ is orthogonal to both $e_i$ and $e_{i+1}$, and $D^*_\delta$ is a portion of a 
disk centered at $e_i\cap e_{i+1}$. Recalling \eqref{eq:sector}, we deduce that $e_i\cap e_{i+1}
\in \Vcal_{\min}$, and the lemma follows.
\item[B.]
The endpoints of $\Gamma$ belong to the same edge $e_i$.
\item[C.]
The endpoints of $\Gamma$ belong to two non-consecutive edges.
\end{enumerate}
The rest of the proof will be devoted to show that cases B and C can not
occur.

In case B, assume w.l.o.g. that $e_i\subset\{(x,0)\in\R^2\}$ and $\Omega\subset\{(x,y)\in\R^2:
y\ge0\}=\Sigma_\pi$. Then $D^*_\delta\cap\Omega = D^*_\delta\cap\Sigma_\pi$, $P(D^*_\delta,\Omega) = P(D^*_\delta,\Sigma_\pi)$, and 
\[
\Rcal(D^*_\delta,\Omega) \ge I(D^*_\delta,\Sigma_\pi) = \sqrt\frac{\pi}{2} > \sqrt\frac{\alpha_{\min}}{2},
\]
in contradiction with \eqref{eq:lemmino}.

Finally, in order to rule out configuration C, by definition of $d$ we have
\[
{\mathcal R}(D^{*}_\delta,\Om)\geq \frac{d}{2\sqrt{|D^{*}_\delta|}}=\frac{d}{2\sqrt{\delta}}> \sqrt\frac{\alpha_{\min}}{2}
\]
whenever $\delta< \bar \delta$, which is fixed as $d/2\alpha_{\min}$.
So that we get again a contradiction concluding the proof.

Finally, the assertion concerning $K(\Om,\bar \delta)$ follows by its definition and from the fact that for all $\delta\leq \bar \delta$ (see also~\cite[Corollary~4.3]{mapeve}), we have just showed that $I(\Om,\delta)=\sqrt{\frac{\alpha}{2}}$ is a constant independent of $\delta$.
\end{proof}
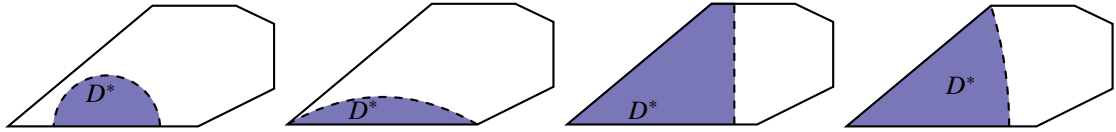
\begin{figure}[ht]
\begin{center}
\hfill
\begin{tikzpicture}
\draw[draw={\grigetto}, fill={\grigetto}] (0,0) -- (2,0) arc (2:178:0.7) -- cycle;
\draw[draw={\grigetto}, fill={\grigetto}] (20:1.3) node {$D^*$};
\draw[thick, dashed] (2,0) arc (2:178:0.7);
\draw[thick] (0,0) -- (1.9,1.6) --  (3,1.6) -- (3.5,1.36) -- (3.5,0.5) -- (2.5,0) -- cycle;
\end{tikzpicture}
\hfill
\begin{tikzpicture}
\draw[draw={\grigetto}, fill={\grigetto}](0,0) to [out=30, in=150] (2.5,0);
%\draw[draw={\grigetto}, fill={\grigetto}] (20:0.3) node {$D^*$};
\node at (1,0.2) {$D^{*}$};
\draw[thick, dashed]  (0,0) to [out=30, in=150] (2.5,0);
\draw[thick] (0,0) -- (1.9,1.6) --  (3,1.6) -- (3.5,1.36) -- (3.5,0.5) -- (2.5,0) -- cycle;
\end{tikzpicture}
\hfill
\begin{tikzpicture}
\draw[draw={\grigetto}, fill={\grigetto}](0,0) -- (1.9,1.61)--(1.9,1.61)--(2.2,1.61)-- (2.2,1.6) --(2.2,0);

\node at (1,0.2) {$D^{*}$};
\draw[thick, dashed] (2.2,0) -- (2.2,1.6);
\draw[thick] (0,0) -- (1.9,1.6) --  (3,1.6) -- (3.5,1.36) -- (3.5,0.5) -- (2.5,0) -- cycle;
\end{tikzpicture}
\hfill
\begin{tikzpicture}
\draw[draw={\grigetto}, fill={\grigetto}] (0,0) -- (2.14,0) arc (0:17:5.45) -- cycle;
\draw[draw={\grigetto}, fill={\grigetto}] (20:1.6) node {$D^*$};
\draw[thick, dashed] (2.14,0) arc (0:17:5.45);
\draw[thick] (0,0) -- (1.9,1.6) --  (3,1.6) -- (3.5,1.36) -- (3.5,0.5) -- (2.5,0) -- cycle;
\end{tikzpicture}
\hfill
\caption{some possibilities for cases B (on the left) and C (on the right) in the proof of Lemma \ref{le:isop}.  The Dirichlet boundary $\partial D^*\cap\Omega$ is dashed.
\label{fig:HO}}
\end{center}
\end{figure}
\begin{remark}
Notice that the threshold $\bar\delta$ in Lemma~\ref{le:isop} has no reason to be optimal. On the 
other hand, one can easily check that in the case of a rectangle, as treated in Theorem 
\ref{thm:orthotope} it is actually optimal, since, for $\delta>\bar\delta$, $I(\Omega,\delta)$ 
is achieved by a rectangle (see e.g. \cite[Remark 4.5]{mapeve}).
\end{remark}
We are now in position to prove our main results.
\begin{proof}[Proof of Theorem~\ref{thm:main1}]
First of all, we take $\eps\in (\delta/\beta,\bar \delta-\delta)\not=\emptyset$ by the assumption on $\delta$ and we apply \cite[Corollary~4.3]{mapeve} and Lemma~\ref{le:isop} to deduce that
\[
\begin{split}
\sd(\delta)&=K^2(\Om,\delta)\delta^{-1}\lambda_{1}^{\text{Dir}}=\alpha_{\min}(2\delta)^{-1}\lambda_{1}^{\text{Dir}}\\
\sd(\delta+\eps)&=K^2(\Om,\delta+\eps)(\delta+\eps)^{-1}\lambda_{1}^{\text{Dir}}=\alpha_{\min}[2(\delta+\eps)]^{-1}\lambda_{1}^{\text{Dir}},
\end{split}
\] 
where
$\lambda_{1}^{\text{Dir}}$ stands for the first eigenvalue of the Dirichlet-Laplacian in the ball of unit radius.
By Theorem \ref{thm:convergence} we obtain
\[
1\geq \frac{\od(\beta,\delta)}{\sd(\delta)}\geq \frac{\sd(\delta+\eps)}{\sd(\delta)}\left(1-\sqrt{\frac{\delta}{\eps\beta}}\right)^{2}
=\frac{\delta}{\delta+\eps}\left(1-\sqrt{\frac{\delta}{\eps\beta}}\right)^{2},
\]
for all $\eps\in(\delta/\beta,\bar \delta-\delta)$.
Then we make the choice of $\eps=\delta/\beta^{1/3}$, which is admissible since
$\beta>1$ and $\delta<\beta^{1/3}\bar{\delta}/(1+\beta^{1/3})$,
and obtain 
\begin{equation}\label{eq:fineprimo}
1\geq \frac{\od(\beta,\delta)}{\sd(\delta)}\geq \frac{1}{1+\beta^{-1/3}}\left(1-\beta^{-1/3}\right)^2,
\end{equation}
yielding the conclusion.
\end{proof}
\begin{proof}[Proof of Theorem~\ref{thm:main2}]
Calling $D^*=B_{r(\delta)}(p)\cap \Omega$, for some $p\in \Vcal_{\min}$ and using conclusion 2 of~\cite[Lemma~3.1]{mapeve}, we infer that  $\la(\beta,D^*)\le \mu(D^*,\Om)$. 
As a consequence we can use Theorem \ref{thm:main1} to write
\[
1\leq \frac{\la(\beta,D^*)}{\Lambda(\beta,\delta)}\leq \frac{M(\delta)}{\Lambda(\beta,\delta)}\leq (1+\beta^{-1/3})\left(1-\beta^{-1/3}\right)^{-2}.\qedhere
\]
%\textcolor{red}{the optimal case is when I consider $g(t_{0})$
%where $g(t_{0})=\max g(t)$ for $t\in (\sqrt{\delta/\beta},\sqrt{|\Omega|-\delta})$
%and $g(t)=(1-\sqrt{\frac{\delta}{\beta}}/t)\delta/(\delta+t^{2})$. Our choice 
%is $t=\sqrt{\delta}/\beta^{1/6}$, 
%}
\end{proof}

\begin{remark}
The estimate of Theorem~\ref{thm:main2} can be read as, 
\[
1\leq \frac{\la(\beta,D^*)}{\Lambda(\beta,\delta)}\leq 1+3\beta^{-1/3}+o(\beta^{-1/3}),\qquad \text{as }\beta\rightarrow\infty.
\]
On the other hand, even without using asymptotic expansions, as $\beta$ increases, the estimate becomes more precise. As an example, for all $\beta>8$, one has the explicit estimate\[
1\leq \frac{\la(\beta,D^*)}{\Lambda(\beta,\delta)}\leq 1+15\beta^{-1/3}+14\beta^{-2/3}.
\]
\end{remark}

\section*{Acknowledgments} 
Work partially supported by the project ERC Advanced Grant 2013 n. 339958:
``Complex Patterns for Strongly Interacting Dynamical Systems - COMPAT'', by the PRIN-2015KB9WPT Grant:
``Variational methods, with applications to problems in mathematical physics and geometry'', and by the INdAM-GNAMPA group.

%\bibliography{frac_eig}
%\bibliographystyle{abbrv}

\medskip
\small
\begin{flushright}
\noindent \verb"dariocesare.mazzoleni@unicatt.it"\\
Dipartimento di Matematica e Fisica ``N. Tartaglia'', Universit\`a Cattolica -- Brescia\\
via dei Musei 41, 25121 Brescia, Italy\\
\medskip
\noindent \verb"benedetta.pellacci@unicampania.it"\\
Dipartimento di Matematica e Fisica,
Universit\`a della Campania  ``Luigi Vanvitelli''\\
viale A. Lincoln 5, Caserta, Italy\\
\medskip
\noindent \verb"gianmaria.verzini@polimi.it"\\
Dipartimento di Matematica, Politecnico di Milano\\
piazza Leonardo da Vinci 32, 20133 Milano, Italy\\
\end{flushright}

\end{document}